\documentclass[12pt]{amsart}

\usepackage[utf8]{inputenc}%
\usepackage[T1]{fontenc}%
\usepackage{lmodern}%
\usepackage{amsxtra,amssymb,amsthm,amsmath,amscd,url,mathrsfs}
\usepackage{mathtools}
\usepackage{fullpage}
\usepackage{color} 
\usepackage[
        colorlinks,
]{hyperref}

\DeclareMathOperator{\IN}{Inner}
\DeclareMathOperator{\wi}{width}

\theoremstyle{plain}
\newtheorem{theorem}{Theorem}[section]
\newtheorem{corollary}[theorem]{Corollary}
\newtheorem{lemma}[theorem]{Lemma}
\newtheorem{observation}[theorem]{Observation}
\newtheorem{conjecture}[theorem]{Conjecture}

\makeatletter
\newcommand{\smallpar}[1]{%
  \par
  \addvspace{\medskipamount}
  \textit{#1\@addpunct{.}}\enspace\ignorespaces
}
\makeatother
\renewcommand{\le}{\leqslant}
\renewcommand{\ge}{\geqslant}
\renewcommand{\leq}{\leqslant}
\renewcommand{\geq}{\geqslant}
\newcommand{\sst}[2]{\left\{#1\,\middle|\,#2\right\}}
\newcommand{\abs}[1]{\left\lvert#1\right\rvert}

\title{Equitable colorings of $K_4$-minor-free graphs}
\thanks{This work falls within the scope of A.N.R. project STINT}
\author{Rémi de Joannis de Verclos}
\address{Laboratoire G-SCOP, C.N.R.S. et Univ. Grenoble Alpes, Grenoble, France.}
\email{remi.de-joannis-de-verclos@grenoble-inp.fr}
\author{Jean-S\'ebastien Sereni}
\address{Centre National de la Recherche Scientifique, LORIA, Vand\oe uvre-l\`es-Nancy,
France.} \email{sereni@kam.mff.cuni.cz}
\date{\today}

\begin{document}

\begin{abstract} 
      We demonstrate that for every positive integer~$\Delta$, every
      $K_4$-minor-free graph with maximum degree~$\Delta$ admits an equitable
      coloring with~$k$ colors where~$k\ge\frac{\Delta+3}{2}$. This bound is tight
      and confirms a conjecture by Zhang and~Whu. We do not use the discharging
      method but rather exploit decomposition trees of $K_4$-minor-free graphs.
\end{abstract}

\maketitle
\section{Introduction}\label{sec:intro} 
Equitable coloring is an ubiquitous notion. From a combinatorial point of view,
it corresponds to a natural variation of usual graph coloring where the color
classes are required to all have the same size, plus/minus one vertex.
Practically, this is one way to prevent color classes from being very large,
which can be useful when using graph coloring for scheduling purposes for instance.
Theoretically, equitable colorings were used successfully in \emph{a priori} unrelated
topics, such as probability. Indeed, one of the seminal results regarding equitable
colorings is the following theorem, which was established by Hajnal and
Szemerédi~\cite{HaSz70} (the statement was first conjectured by Erd\H{o}s).
\begin{theorem}[Hajnal--Szemerédi, 1970]
      Every graph with maximum degree at most~$\Delta$ admits an equitable
      coloring using~$\Delta+1$ colors.
      \label{thm:hasz}
\end{theorem}
Theorem~\ref{thm:hasz} allowed for a simplified demonstration of the Blow-up lemma
--- found by R\"odl and Ruci\'nski~\cite{RoRu99}. In addition, this theorem was
also used to derive deviation bounds for sums of random variables with some degree
of dependence --- this was done by Alon and F\"uredi~\cite{AlFu92} and by Janson
and Ruci\'nski~\cite{JaRu02}.  Let us point out that in~2010, that is, forty years
after Theorem~\ref{thm:hasz} was proved, a much simpler demonstration was finally
found, building on several other related results. More precisely, Kierstead,
Kostochka, Mydlarz and Szemerédi~\cite{KKMS10} managed to find a two-page proof of
Theorem~\ref{thm:hasz}, which also has the advantage to lead to a polynomial-time
algorithm that efficiently finds a relevant coloring --- contrary to the original
argument.

As it turns out, the notion of equitable colorings behaves pretty differently from
usual colorings, and it is a challenging task to better comprehend its relation to
well-known graph classes.  Starting from graphs with bounded maximum degree, it is
natural to consider next $d$-degenerate graphs. The following theorem was
established by Kostochka and Nakprasit~\cite{KoNa03}, in a more general form.
\begin{theorem}[Kostochka--Nakprasit, 2003]
      Let $\Delta$ be an integer greater than~$53$.
     If $G$ is a $2$-degenerate graph with maximum degree at most~$\Delta$, then
     $G$ is equitably $k$-colorable whenever $k\ge\frac{\Delta+3}{2}$.
      \label{thm:kona}
\end{theorem}
Theorem~\ref{thm:kona} partially confirms a conjecture by Zhang and
Wu~\cite[Conjecture~9]{ZhWu11}, (also see~\cite[Conjecture~6, p.~1209]{Lih13}) that if $\Delta\ge3$, then every series-parallel
graph with maximum degree~$\Delta$ admits an equitable $k$-coloring whenever
$k\ge\frac{\Delta+3}{2}$. Indeed, series-parallel graphs are known to be~$2$-degenerate,
so Theorem~\ref{thm:kona} yields that the conjecture is true if~$\Delta\ge54$.
The purpose of our work is to establish the conjecture for all the remaining
cases, that is, $\Delta\in\{3,\dotsc,53\}$. (Although, in our proofs we do not use
the upper bound on~$\Delta$, and simply prove the statement for all
$K_4$-minor-free graphs.)

The statement conjectured by Zhang and Wu is actually a strengthening of a result of theirs~\cite{ZhWu11},
which establishes that every series-parallel graph with maximum degree~$\Delta\ge3$ admits
an equitable $k$-coloring if $k\ge\Delta$. The conjecture can also be seen
as a generalisation of a theorem of Kostochka~\cite{Kos02} that every outerplanar
with maximum degree~$\Delta\ge3$ admits an equitable $k$-coloring whenever
$k\ge\frac{\Delta+3}{2}$.

It is worth mentioning that Kostochka, Nakprasit and Pemmaraju~\cite{KNP05}
established (a generalisation of) the following interesting statement.
\begin{theorem}[Kostochka, Nakprasit \& Pemmaraju, 2005]
     Fix an integer~$k\ge124$. If $G$ is a $2$-degenerate graph
     with maximum degree at most~$\frac12\abs{V(G)}+1$, then $G$ admits
     an equitable $k$-coloring.
      \label{thm:knp}
\end{theorem}
Theorem~\ref{thm:knp}, however, does not bring us any new information regarding
the problem at hands. Indeed, we need to consider graphs with maximum
degree~$\Delta\le53$, while the number of colors needs to be at least~$124$. Hence
for our question the information provided by Theorem~\ref{thm:knp} is already
contained in the aforementioned result of Zhang and Wu~\cite{ZhWu11}.

As reported earlier, we establish the following.
\begin{theorem}
      If $G$ is a $K_4$-minor-free graph with maximum degree~$\Delta$, then
      $G$ admits an equitable $k$-coloring whenever $k\ge\frac{\Delta+3}{2}$.
      \label{thm:main}
\end{theorem}
Contrary to the proof of some of the results mentioned above, we do not rely on
discharging, but rather on the structural links between $K_4$-minor-free graphs
and two-terminal series-parallel graphs: in particular, our proof heavily relies
on a so-called SP-tree.  Before proceeding with the proof, we review some folklore
properties of $K_4$-minor-free graphs and two-terminal series-parallel graphs and
introduce a bit of terminology.

It would be interesting to know whether Theorem~\ref{thm:main} can be extended to the class
of $2$-degenerate graphs. A generalisation of this has actually been conjectured
in~2003 by Kostochka and Napkrasit~\cite{KoNa03}.
\begin{conjecture}
      Fix an integer~$\Delta$. If~$d\in\{2,\dotsc,\Delta\}$ and~$G$ is a $d$-degenerate
      graph with maximum degree at most~$\Delta$, then $G$ admits an equitable $k$-coloring
      whenever $k\ge\frac{\Delta+d+1}{2}$.
      \label{conj:degenerate}
\end{conjecture}

\section{The structure of $K_4$-minor-free Graphs}\label{sec:struct} 
As it turns out, graphs with no $K_4$-minor are strongly related to two-terminal
series-parallel graphs.  A \emph{two-terminal graph} is a graph with two
distinguished vertices called \emph{poles}.  \emph{Two-terminal series-parallel
graphs} are two-terminal graphs that can be obtained by the following recursive
construction\footnote{We point out that in the literature, such graphs are
      sometimes called simply 'series-parallel graphs', while this term can also
      be used to refer to $K_4$-minor-free graphs.}.
The basic two-terminal series-parallel graph is an edge~$uv$ with the two poles
being its end-vertices. For~$i\in\{1,2\}$, let~$G_i$ be a two-terminal
series-parallel graph with poles~$u_i$ and~$v_i$. The graph~$S(G_1,G_2)$ obtained
by identifying the vertices~$v_1$ and~$u_2$ is also a two-terminal series-parallel
graph and its two poles are the vertices~$u_1$ and~$v_2$. The graph~$S(G_1,G_2)$
obtained in this way is called the \emph{serial join} of~$G_1$ and~$G_2$. The
\emph{parallel join} of~$G_1$ and~$G_2$ is the graph~$P(G_1,G_2)$ obtained by
identifying the pairs of vertices~$(u_1,u_2)$ and~$(v_1,v_2)$; the poles
of~$P(G_1,G_2)$ being the identified vertices. Two-terminal series-parallel graphs
are precisely those that can be obtained from edges by a series of serial and
parallel joins.  The decomposition tree corresponding to a two-terminal
series-parallel graph~$G$ is not unique.  In fact, there is a lot of freedom in
its choice as can be seen in the following well-known result.
\begin{lemma}\label{v-decomp} Let~$G$ be a two-terminal series-parallel graph
      and~$v$ a vertex of~$G$. There exists an SP-decomposition tree such that $v$
      is one of the poles of the graph corresponding to the root of the
      SP-decomposition tree.
\end{lemma}
It is also well known that every $2$-edge-connected $K_4$-minor-free graph is
a two-terminal series-parallel graph.
\begin{lemma}\label{block0} Every block of a $K_4$-minor-free graph is
      a two-terminal series-parallel graph.
\end{lemma}
The set of $K_4$-minor-free graphs can also be seen as the closure of two-terminal
series-parallel graphs by the spanning subgraph relation.
\begin{lemma}\label{block} A graph~$G$ has no $K_4$-minor if and only if $G$ is
  the spanning subgraph of a two-terminal series-parallel graphs.
\end{lemma}
To see Lemma~\ref{block}, note that spanning subgraphs of two-terminal series-parallel graphs has no $K_4$-minor. The reversed direction is deduced by induction
on the structure of $K_4$-minor-free graph given by Lemma~\ref{block0}
using Lemma~\ref{v-decomp}.

As a consequence, the $K_4$-minor-free graphs are precisely those for which we can
choose two poles such that the two-terminal graph obtained can be constructed from
the two graphs of size two by a series of serial and parallel joins.  The
construction of a particular $K_4$-minor-free graph~$G$ can thus be encoded by
a rooted tree, which is called the \emph{SP-decomposition tree} of~$G$. Each node
of the tree corresponds to a subgraph of~$G$ obtained at a step of the recursive
construction of~$G$. The leaves correspond to graphs with only two poles (and no
other vertex) that may or may not be connected by an edge.  Each inner node of
the tree corresponds to either a serial join or to a parallel join.  Based on
this, there are two types of inner nodes: \emph{S-nodes} and \emph{P-nodes}. The
inner nodes have at least two children: the subgraphs corresponding to their
children are joined together by a sequence of serial or parallel joins depending
on the type of the node. Since the result of a sequence of serial joins depends on
the order in which the serial joins are applied, the children of each inner node
are ordered. Without loss of generality, we can assume that the children of
a P-node are S-nodes and leaves only, and the children of an S-node are P-nodes
and leaves only.

If~$A$ is a two-terminal graph,
the vertices of~$A$ distinct from its poles are said to be its \emph{inner}
vertices. The set of inner vertices of~$A$ is $\IN(A)$.  We define~$\wi(A)$, the
\emph{width} of~$A$, to be the number of inner vertices of~$A$, that is,
$\wi(A)=\abs{\IN(A)}$ (note that $\wi(A)=\abs{V(A)}-2$).  We introduce some
terminology for particular two-terminal $K_4$-minor-free graphs.  A two-terminal
graph obtained by a parallel join of several two-edge paths is a \emph{diamond}.
A two-terminal graph obtained by a parallel join of several two-edge paths and an
edge is a \emph{crystal}. Observe that an edge may be seen as a crystal of
width~$0$. If $i$ is a positive integer, we define~$D(i)$ to be the diamond with
width~$i$ and~$C(i)$ to be the crystal with width~$i$. Let~$D'(1)$ be the
graph~$K_{1,3}$ with two vertices of degree~$1$ as poles. For~$i \geq 2$,
we define~$D'(i)$ to be the graph obtained by a parallel join of~$D'(1)$
with $i-1$ paths of length~$2$. Let~$C'(i)$ be obtained from~$D'(i)$ by adding an
edge between the poles.  We let~$P_i$ be the path with~$i$ vertices.  If $G$ is
a graph and~$U$ a subset of the vertices of~$G$, we let~$G-U$ be the subgraph
of~$G$ induced by the vertices of~$G$ that do not belong to~$U$. For a positive
integer~$k$, we take the representatives of~$\mathbf{Z}_k$ to be~$\{1,\dotsc,k\}$,
rather than the more common~$\{0,\dotsc,k-1\}$.  An \emph{equitable $k$-coloring}
of a graph~$G$ is a mapping~$\alpha\colon V(G)\to\mathbf{Z}_k$ such that
$\abs{\alpha^{-1}(\{i\})}$ and~$\abs{\alpha^{-1}(\{j\})}$ differ by at most one for every
$(i,j)\in\mathbf{Z}_k^2$.

The next lemma is a simple but useful remark about common neighbors of the poles
of a $K_4$-minor-free graph.
\begin{lemma}
  If~$H$ is a $K_4$-minor-free graph with poles~$a$ and~$b$, then 
  $N_H(a)\cap N_H(b)$ is an independent set of~$H$.
  \label{lem:independent}
\end{lemma}
\begin{proof}
  We prove by induction on the number of vertices of the SP-decomposition tree of~$H$ that
  no two vertices in~$N_H(a)\cap N_H(b)$ belong to a same component
  of~$H\setminus\{a,b\}$.
  \begin{itemize}
  \item The statement is trivial if the SP-tree has only one node, that is if $H$
        has two vertices.
  \item If $H$ is the series join of~$H_1$ and~$H_2$, then the only possible
        common neighbor of~$a$ and~$b$ is the common pole of~$H_1$ and~$H_2$. The
        statement is therefore true in this case also.
  \item If~$H$ is the parallel join of~$H_1$ and~$H_2$, then let $x$ and~$y$ be
        two common neighbors of~$a$ and~$b$. Either $x$ and~$y$ belong to~$H_i$
        for some~$i\in\{1,2\}$, in which case the result follows from the
        induction hypothesis applied on~$H_i$; or $x$ and~$y$ are in different
        components of~$H \setminus \{a,b\}$.
    \end{itemize}
\end{proof}

Let~$T$ be an SP-decomposition tree (of a $K_4$-minor-free graph),
and~$n$ be a node of~$T$ representing
the subgraph~$H$ with poles~$a$ and~$b$. Assume that $H-\{a,b\}$
has~$m$ components~$C^1,\dotsc,C^m$.
The node~$n$ is in \emph{normal form} if $m\le1$ (\emph{i.e.}~$H-\{a,b\}$ either is connected or
has no vertex at all), or
if $n$ is a parallel node with children~$H^1,\dotsc,H^m$ plus the edge~$ab$
if $ab \in E(H)$,
where $H^i$ is the subgraph of~$H$ induced by~$C^i\cup\{a,b\}$ from which we
remove the edge~$ab$ if it is present.
The tree~$T$ is in \emph{normal form} if every node of~$T$ is in
normal form.
\begin{lemma}\label{lem:decomposition}
  If~$G$ is a $K_4$-minor-free graph,
  then~$G$ admits a construction tree~$T$ in normal form.
\end{lemma}
\begin{proof}
  As a $K_4$-minor-free graph, $G$ has two vertices~$a$ and~$b$ and
  an SP-decomposition tree~$T$ that represents the two-terminal graph~$G$
  with poles~$a$ and~$b$. Note that we may assume that~$T$ is a binary tree
  (where P-nodes and S-nodes may not alternate).

  To prove the lemma,
  we describe an inductive procedure that transform the (binary) SP-decomposition
  tree~$T$ into an SP-decomposition tree~$T'$ in normal form
  that represents the same graph~$G$.
  Assume that this procedure exists for trees with fewer nodes than~$T$.
  If~$n$ is a leaf, then~$G$ has two vertices and further
    $V(G)-\{a,b\}$ is empty, so~$n$ is in normal form indeed.
  So we now suppose that $n$ has two children representing the graphs~$G_1$
  and~$G_2$, respectively. By induction, for each $i\in\{1,2\}$ there is
  a tree~$T_i$ in normal form that represents~$G_i$.
We distinguish two cases depending on the type of the root~$n$ of~$T$.
  \begin{itemize}
  \item Suppose that $n$ is a P-node, so $G=P(G_1,G_2)$.
    Let~$C_i^1,\dotsc,C_i^{m_i}$ be the components of~$G_i-\{a,b\}$, and note
    that $m_i$ is a positive integer.
    If~$m_i=1$, then we set $H_i^1\coloneqq G_i$.
    If~$m_i \geq 2$, then according to the definition of normal forms
    the graph $G_i$ is encoded in~$T_i$
    by the parallel join of $H_i^1,\dots,H_i^{m_i}$, plus possibly the edge~$ab$.
    (We recall that it means that each graph~$H_i^j$ is
    the subgraph of~$G_i$ induced by $C_i^j\cup\{a,b\}$ from which the edge~$ab$ is deleted
    if it is present.)
    The sought SP-decomposition tree~$T'$ is then obtained by making
    a new P-node~$n$ the parent node of each of the SP-decomposition
    trees representing $H_1^1,\dots,H_1^{m_1},H_2^1,\dots,H_2^{m_2}$ (each of them
    in normal form), and, possibly,
    of a leaf representing an edge if $ab \in E(G)$.

  \item Suppose that $n$ is an S-node, so $G=S(G_1,G_2)$.
    First note that $ab \notin E(G)$.
    Let~$c$ be the common pole of~$G_1$ and~$G_2$.
    Let~$C_1^1,\dotsc,C^{k_1}_1$ be the components
    of~$G_1-\{a,c\}$ that contain a neighbor of~$c$ and
    let~$C^{k_1+1}_1,\dotsc,C^{m_1}_1$ be the other components of
    $G_1-\{a,c\}$. We define analogously the components~$C_2^1,\dotsc,C^{m_2}_2$
    and the index~$k_2$ with respect to~$G_2-\{b,c\}$.
    For each~$j\in\{1,\dotsc,m_1\}$, we define~$H_1^j$ to be the subgraph of~$G$ corresponding to the
    component~$C_1^j$ of~$G_1-\{a,c\}$ as in the definition of normal forms.
    The graphs~$H_2^1,\dotsc,H_2^{m_2}$ are defined analogously
    with respect to~$G_2-\{b,c\}$.

    According to the definition of normal forms, either $H_i^1=G_i$ or,
    in~$T_i$, the graph~$G_i$ is represented by~$P(H_i^1,\dotsc,H_i^{m_i})$.
    Note that the components of~$G-\{a,b\}$ are exactly
    $C^{k_1+1}_1,\dotsc,C^{m_1}_1,C^{k_2+1}_2,\dotsc,C^{m_2}_1$ and
    $\{c\}\cup(\bigcup_{j=1}^{k_1}C^j_1)\cup(\bigcup_{j=1}^{k_2}C^j_2)$.
    Based on this, the sought tree~$T'$ is the tree with a P-node as a root,
    whose children are the SP-decomposition trees
    representing~$H_1^{k_1+1},\dotsc,H_1^{m_1},H_2^{k_2+1},\dotsc,H_2^{m_2}$
    and~$S(F_1,F_2)$, where~$F_i\coloneqq P(H_i^1,\dotsc,H_i^{k_1})$
    for~$i\in\{1,2\}$ (each of them in normal form).
    It follows from the construction that the node~$n$ is in normal form,
    hence so is the tree~$T'$. This concludes the proof.
\end{itemize}
\end{proof}

\section{Reductions}\label{sec:red} 
We note that the statement of Theorem~\ref{thm:main} is true if $k\le2$, since
then $\Delta\in\{0,1\}$. So from now on we assume that $k\ge3$.
We fix a minimal counter-example~$(G,k)$, where $k\ge\lceil\frac{\Delta(G)+3}{2}\rceil$,
along with an SP tree-decomposition~$T$ of~$G$ with every node in normal form
(Lemma~\ref{lem:decomposition} ensures that this is possible).
It follows that $k<\abs{V(G)}$, as any graph~$H$
admits an equitable $t$-coloring if $t\ge\abs{V(H)}$.  We may also assume that $G$
is connected. As a consequence, every component of a subgraph of~$G$
with poles~$a$ and~$b$ that is represented by a subtree of~$T$ contains~$a$
or~$b$. A subtree~$T'$ of~$T$ is a \emph{construction subtree} if $T'$ is rooted
at a node~$r$ of~$T$ and $T'-\{r\}$ consists of at least two subtrees of~$T-\{r\}$ containing
children of~$r$ such that if~$r$ is an $S$-nodes, then all these children
are consecutive around~$r$ in~$T$.

Throughout this section, each time a coloring~$c$ is obtained by
induction (or, equivalently, by a minimality argument), we assume the colors to be
ordered increasingly, that is, such that $\abs{\alpha^{-1}(\{i\})} \leq
\abs{\alpha^{-1}(\{j\})}$ for every two colors~$i$ and~$j$ with~$i<j$. (This condition
implies that if we consider a $k$-coloring~$\alpha$ of an $n'$-vertex graph with
$n'<k$, then the colors used by~$c$ are precisely~$k,k-1,\dotsc,k-n'$, each being used
exactly once.)
\begin{lemma}
  The graph~$G$ has no construction subtree representing a subgraph~$C(k-1)$ or~$D(k-1)$.
  \label{lem:treeC}
\end{lemma}
\begin{proof}
  Suppose, on the contrary, that $H$ is such a subgraph of~$G$.  Let~$a$ and~$b$
  be the poles and~$v_1,\dotsc,v_t$ the inner vertices of~$H$.  Let~$F$ be the
  graph constructed from~$G$ by contracting~$V(H)$ to a vertex~$c$, removing
  parallel edges and loops when they occur. Note that $F$ has no $K_4$-minor. In
  addition, $d_F(c) \leq d_G(a) + d_G(b) - 2(k-1) \leq 2k-4$.  By the minimality
  of~$G$, there is an equitable $k$-coloring~$\alpha$ of~$F$.
  Define~$\alpha'(v)\coloneqq\alpha(v)$ for~$v \in V \setminus V(H)$.  Note that
  $\alpha'$ is a partial proper coloring of~$G$, that is, a proper coloring
  defined on a subset of~$V(G)$.  To finish the proof, it suffices to
  extend~$\alpha'$ to a proper coloring of~$G$ such that the
  multisets~$\{\alpha'(a),\alpha'(b),\alpha'(v_1),\dotsc,\alpha'(v_k)\}$
  and~$\{\alpha(c),1,\dotsc,k\}$ are equal. (Note that in this latter multiset one
  color has multiplicity two --- namely $\alpha(c)$ --- and $k-1$ colors have
  multiplicity one.) We now distinguish two cases.
  \begin{itemize}
  \item If $ab \notin E$, then we
        set~$\alpha'(a)\coloneqq\alpha'(b)\coloneqq\alpha(c)$ and we
        color~$v_1,\dotsc,v_k$ using all the elements of the
        set~$\{1,\dotsc,k\}\setminus\{\alpha(c)\}$.
  \item If $ab \in E$, then $a$ has at most~$\Delta - k \leq k-3$ colored
        neighbors. So $a$ can be properly colored with a color~$\alpha'(a)$
        different from~$\alpha(c)$. Similarly, $b$ has at most~$k-2$ colored
        neighbors (including~$a$), so $b$ can be properly colored with
        a color~$\alpha'(b)$ different from~$\alpha(c)$ (and from~$\alpha'(a)$).
        Now, we color~$v_1,\dotsc,v_k$ using the elements of the
        multiset~$\{\alpha(c),\alpha(c),1,\dotsc,k\}\setminus\{\alpha'(a),\alpha'(b)\}$,
        with the corresponding multiplicities.
  \end{itemize}
\end{proof}
\begin{corollary}
  For every integer~$t\geq k-1$,
  the graph~$G$ has no construction subtree
  representing a subgraph~$C(t)$ or~$D(t)$.
  \label{lem:treeC:2}
\end{corollary}
\begin{proof}
  Assume otherwise that~$H$ is such a subgraph of~$G$.
  Let~$a$ and~$b$ be the poles of~$H$.
  Let~$n$ be the root of the construction subtree that represents~$H$.
  Since~$n$ is in normal form and $H-\{a,b\}$ is an independent set of size~$t$,
  the node~$n$ is a parallel node with at least~$t$ children representing a
  path~$P_3$
  with end vertices~$a$ and~$b$ (the node~$n$ may have other children as well).
  Choosing~$n$ as a root along with~$k-1$ of the children of~$n$ representing
  a~$P_3$ yields a construction subtree of~$T$ that represents~$D(k-1)$,
  which contradicts Lemma~\ref{lem:treeC}.
\end{proof}
\begin{lemma}
  If a construction subtree of~$T$ represents a graph~$H$ with $1\leq\wi(H)\leq k$,
  then $\IN(H)$ is dominated by a pole of~$H$
  unless $\wi(H)\ge2$ and $H\in\{C'(t),D'(t)\}$, where~$t=\wi(H)-1$.
  \label{lem:dominates}
\end{lemma}
\begin{proof}
  Assume that each of the poles~$a$ and~$b$ of~$H$
  has a non-neighbor in~$\IN(H)$, which we name~$a'$ and~$b'$, respectively.
  Note that it is possible to ensure that $a'\neq b'$ unless
  $\IN(H)\setminus N(a)=\IN(H)\setminus N(b)=\{a'\}$.
  In this latter case, since each component of~$H$ contains~$a$ or~$b$ as reported
  earlier, we deduce that $H$ is connected. It then follows from
  Lemma~\ref{lem:independent} that~$H$ is equal to either~$C'(t)$ or~$D'(t)$, with
  $t=\wi(H)-1\ge1$.

  We now assume that $a'\neq b'$, which yields to a contradiction. Indeed,
  let~$F$ be the graph~$G-\IN(H)$ to which we add the edge~$ab$ if it is
  not already present. By the minimality of~$G$ there is an equitable
  $k$-coloring~$\alpha$ of~$F$. To obtain a contradiction, it suffices
  to extend~$\alpha$ to a proper coloring of~$G$ such
  that~$\sst{\alpha(v)}{v\in\IN(H)}$ equals~$\{1,\dotsc,\wi(H)\}$.
  (We recall that the colors are increasingly ordered.)
  
  To do so, we define~$\alpha(a')\coloneqq\alpha(a)$ if $\alpha(a)\leq\wi(H)$ and
  $\alpha(b')\coloneqq\alpha(b)$  if $\alpha(b)\leq\wi(H)$ and we arbitrarily
  assign the colors of $\{1,\dotsc,\wi(H)\}\setminus\{\alpha(a),\alpha(b)\}$ to
  the non-colored vertices, each color being assigned once.
\end{proof}
Our next statement is a direct consequence of~Lemma~\ref{lem:dominates}.
\begin{corollary}
  If a construction subtree of~$T$ represents a graph~$H$ with~$\wi(H)\leq k$,
  then the subgraph induced by~$\IN(H)$ is a forest.
  \label{cor:treetree}
\end{corollary}
\begin{proof}
    The statement is clear if $H\in\{C'(t),D'(t)\}$ for some integer~$t$, so by
    Lemma~\ref{lem:dominates} we can assume that $\IN(H)$ is dominated by
    a pole~$a$ of~$H$.  Then $\IN(H)$ induces an acyclic graph, as otherwise
    $\IN(H)\cup\{a\}$ would induce a subgraph of~$G$ containing a subdivision
    of~$K_4$.
\end{proof}
\begin{lemma}
  Let~$H$ be a graph with poles~$a$ and~$b$ represented by a construction subtree of~$T$
  and assume that $\wi(H)=k-1$. Then $d_H(a)+d_H(b)\leq 2k-4$.
  \label{lem:k-1:contraction}
\end{lemma}
\begin{proof}
  Assume on the contrary that $d_H(a)+d_H(b) \geq 2k-3$.
  Let~$F$ be the graph obtained from~$G$ by contracting~$H$ into one vertex~$c$, again removing parallel edges
  and loops when they occur.
  In other words, we set~$V(F)\coloneqq(V(G)\setminus V(H))\cup\{c\}$ and
  $N_F(v)\coloneqq N_G(v)$ for~$v \in V(G)\setminus V(H)$ while $N_F(c)\coloneqq
  (N_{G}(a)\cup N_{G}(b))\cap V(F)$.  By our assumption, $d_G(c) \leq
  d_G(a)-d_H(a)+d_G(b)-d_H(b)\leq 2\Delta -(2k-3) \leq \Delta$. Consequently, $F$ is
  a $K_4$-minor-free graph with maximum degree at most~$\Delta$. By the minimality of~$G$
  there is an equitable $k$-coloring~$\alpha$ of~$F$.
  To obtain an equitable colouring of~$G$, it suffices to extend~$\alpha$ to~$V(G)$
  in such a way that the multisets~$\sst{\alpha(v)}{v\in V(H)}$
  and~$\{\alpha(c),1,\dotsc,k\}$ are equal. We note that Corollary~\ref{cor:treetree}
  yields that $\IN(H)$ induces an acyclic graph. We distinguish three cases.
  \begin{itemize}
  \item If $ab\notin E(G)$ then we define $\alpha(a)\coloneqq\alpha(b)\coloneqq\alpha(c)$ and we
        arbitrarily distribute all the colors in~$\{1,\dotsc,k\}\setminus\{\alpha(c)\}$
        to the vertices in~$\IN(H)$.

  \item If $ab\in E(G)$ and $a$ has a non-neighbor~$a' \in \IN(H)$, then by
        Lemma~\ref{lem:dominates}, it follows that either~$b$ dominates~$\IN(H)$
        or $H=C'(k-2)$. In both cases, we know that $b$ has at least~$k-2$
        neighbors in~$\IN(H)$. It follows that $b$ has at most~$\Delta-(k-2)\leq
        k-1$ neighbors outside of~$\IN(H)$, including~$a$. We define
        $\alpha(a)\coloneqq\alpha(a')\coloneqq\alpha(c)$. By the preceding remark it is possible
        to properly color~$b$ with a color~$\alpha(b)$ (so in particular
        $\alpha(a)\neq\alpha(b)$). To finish the coloring, we assign
        arbitrarily all the colors in~$\{1,\dotsc,k\}\setminus\{\alpha(a),\alpha(b)\}$ to the
        vertices in~$\IN(H)\setminus\{a'\}$.
      \item If both~$a$ and~$b$ dominate $\IN(H)$,
        then by Lemma~\ref{lem:independent} we know that $H=C(k-1)$, which does not occur by
        Lemma~\ref{lem:treeC}.
  \end{itemize}
\end{proof}
\begin{lemma}
  If~$H$ is a graph represented by a construction subtree of~$G$, then $\wi(H) \neq k-1$.
  \label{lem:tree:k-1}
\end{lemma}
\begin{proof}
  Assume otherwise that there is such a graph~$H$ with width~$k-1$.
  By Lemma~\ref{lem:dominates}, we may assume that a pole~$a$ of~$H$ has at
  least~$k-2$ neighbors in~$\IN(H)$. Let~$b$ be the other pole of~$H$. By
  Lemma~\ref{lem:k-1:contraction}, we have $d_H(b) \leq 2k-4-d_H(a) \leq k-2$. It follows
  that~$b$ has a non-neighbor~$b'$ in~$\IN(H)$.  By the minimality of~$G$, the graph
  $F\coloneqq G-(\IN(H)\cup\{a\})$ has an equitable $k$-coloring~$\alpha$.
  To finish the proof, it suffices to extend~$\alpha$ to~$V(G)$ in such a way that
  $\sst{\alpha(v)}{v\in\IN(H)\cup\{a\}}$ equals~$\{1,\dotsc,k\}$.  Since $a$ has
  at most~$k-1$ colored neighbors, it is possible to properly color~$a$ with
  a color~$\alpha(a)$. We set~$\alpha(b')\coloneqq\alpha(b)$ unless
  $\alpha(a)=\alpha(b)$.  Then we arbitrarily color the ($k-1$ or $k-2$)
  non-colored vertices using all the ($k-1$ or~$k-2$) colors
  in~$\{1,\dotsc,k\}\setminus\{\alpha(a),\alpha(b)\}$.
\end{proof}
\begin{corollary}
  If~$H$ is a graph represented by a construction subtree of~$G$,
  then $H\notin\{C'(k-1),D'(k-1)\}$.
  \label{lem:treeC'}
\end{corollary}
\begin{proof}
  Assume otherwise that~$H$ is such a graph, with poles~$a$ and~$b$,
  and represented by a construction subtree of~$G$ with root~$n$.
  Since~$n$ is in normal form and $H-\{a,b\}$ is disconnected,
  the node~$n$ is a parallel node with a children representing
  a star~$K_{1,3}$ and (at least)~$k-2$ children each representing a
  path~$P_3$ with end-vertices~$a$ and~$b$ (the node~$n$ may have further children).
  It follows that~$T$ has a construction subtree of~$G$ rooted on~$n$
  representing~$D'(k-2)$, which has width~$k-1$.
  This contradicts Lemma~\ref{lem:tree:k-1}.
\end{proof}
\begin{lemma}
  If~$H$ is a graph represented by a construction subtree of~$G$, then $\wi(H) \neq k$.
  \label{lem:tree:k}
\end{lemma}
\begin{proof}
  Suppose, on the contrary, that $H$ is such a graph with width~$k$.
  Let~$a$ and~$b$ be the poles of~$H$.
  By Lemmas~\ref{lem:treeC:2} and~\ref{lem:treeC'}, we know that
  $H\notin\{C(k),C'(k-1),D(k),D'(k-1)\}$. It now follows from
  Lemma~\ref{lem:dominates}, that $a$ dominates~$\IN(H)$. Then
  $b$ has a non-neighbor~$b'\in\IN(H)$, for otherwise
  $b$ also would dominate~$\IN(H)$, so Lemma~\ref{lem:independent} would imply
  that $H \in\{C(k),D(k)\}$.
  
  Let~$F$ be the graph~$G-\IN(H)$ to which we add the edge~$ab$ if it is
  not already present. By the minimality of~$F$ there is an equitable
  $k$-coloring~$\alpha$ of~$F$. To finish the proof, it suffices to
  deduce a proper coloring~$\alpha'$ of~$G$ that equals~$\alpha$
  on~$V(G)\setminus(\IN(H)\cup\{a\})$ and such that the
  multisets~$\sst{\alpha'(u)}{u\in\IN(H)\cup\{a\}}$
  and~$\{1,\dotsc,k\}\cup\{\alpha(a)\}$ are equal.  We distinguish two cases
  depending on the value of~$k$.
  \begin{itemize}
  \item Case~1: $k\geq 4$.  Since $a$ has~$k$ neighbors in~$\IN(H)$, the
        vertex~$a$ has at most $\Delta - k \leq k - 3$ colored neighbors, so we
        can properly recolor~$a$ with a color~$\alpha'(a)$ different from both~$\alpha(a)$
        and~$\alpha(b)$. By Corollary~\ref{cor:treetree}, $\IN(H)$ is a forest and we
        know that $\abs{\IN(H)\setminus\{b'\}}=k-1\geq 3$, so there is an independent
        set $A\subset\IN(H)\setminus\{b'\}$ of size~$2$.  To complete the
        coloring, we assign $\alpha(b)$ to~$b'$ and~$\alpha(a)$ to the
        vertices in~$A$ and we distribute arbitrarily the colors
        in~$\{1,\dotsc,k\}\setminus\{\alpha'(a),\alpha(a),\alpha(b)\}$ to the non-colored
        vertices.
  \item Case~2: $k=3$.  Since~$a$ dominates a set of size~$k$, it holds that
        $k \leq \Delta \leq 2k-3$, so $k=3=\Delta$. Moreover, it also follows
        that $ab \notin E$.  As a consequence of Corollary~\ref{cor:treetree}, the
        set~$\IN(H)$ contains two non-adjacent vertices~$v_1$ and~$v_2$. Let~$u$
        be the third vertex in~$\IN(H)$, so $\IN(H)=\{v_1,v_2,u\}$. We
        define~$\alpha'(a)\coloneqq\alpha(b)$, we
        set~$\alpha'(v_i)\coloneqq\alpha(a)$ for~$i\in\{1,2\}$ and we attribute
        to~$u$ the third color, that is the one
        in~$\{1,2,3\}\setminus\{\alpha(a),\alpha(b)\}$. 
  \end{itemize}
  In both cases, we obtain an equitable $k$-coloring of~$G$, a contradiction. 
\end{proof}
Our last two lemmas rely on the following observation.
\begin{observation}
      Let~$m$ be a positive integer and let~$\lambda_1,\dotsc,\lambda_m\in\{1,2\}$.
      If $A_1$ and~$A_2$ are two subsets of the vertices of a graph~$G$ that has no edge
      between~$A_1$ and~$A_2$, then the vertices in~$A_1\cup A_2$ can be properly
      colored using the colors $1,\dotsc,m$ with respective multiplicities
      $\lambda_1,\dotsc,\lambda_m$ whenever
      $\sum_{j=1}^m\lambda_j=\abs{A_1}+\abs{A_2}$ and $\abs{A_i}\leq m$ for~$i\in\{1,2\}$.
  \label{obs:2parts}
\end{observation}
\begin{proof}
      For~$s \in \{1,2\}$, set~$m_s\coloneqq\sst{i\in\{1,\dotsc,m\}}{\lambda_i=s}$.
      We know that $\abs{A_1}+\abs{A_2}=m_1+2m_2=m+m_2$. We deduce that $A_1\leq m_2$ and $A_2\leq m_2$.
      This ensures that the following greedy procedure is valid.
      For every color~$i$ with $\lambda_i=2$, we color one vertex in~$A_1$ and
      one vertex in~$A_2$ with~$i$. After that, it remains to assign
      arbitrarily the $m_1$ colors of multiplicity~$1$ to the~$m_1$ non-colored
      vertices. 
\end{proof}
\begin{lemma}
  Let~$H\coloneqq P(H_1,H_2)$ be a graph represented by a construction subtree of~$T$. Assume
  that $\wi(H_i)\leq k-2$ for~$i \in \{1,2\}$.  Then $\wi(H) \leq k-2$.
  \label{lem:parallel} 
\end{lemma}
\begin{proof}
  We proceed by contradiction. Let~$H$ be a minimal counter-example.  By
  Lemmas~\ref{lem:tree:k-1} and~\ref{lem:tree:k}, we know that $\wi(H) =  k + \mu$
  for some positive integer~$\mu$.
  
  Let~$a$ and~$b$ be the poles of~$H$.  We first prove that every
  component~$U$ of~${H-\{a,b\}}$ has at least~$\mu+2$ vertices.
  Indeed, since the root~$n$ of the construction subtree representing~$H$
  is in normal form,
  the node~$n$ is a parallel node and the subgraph induced by~$U\cup\{a,b\}$,
  from which we
  remove the edge~$ab$ if it is present, is represented by a children of~$n$,
  so $H'\coloneqq H-U$ is represented by a construction subtree of~$T$.
  If moreover $\abs{U}\leq\mu+1$, then~$H'$ has width at least~$k-1$,
  thereby contradicting the minimality of~$H$.
  In particular, $\wi(H_i)\geq \mu+2\geq 3$ for $i\in\{1,2\}$,
  so $k\geq 5$.
  
  Assume for the time being that neither~$a$ nor~$b$ dominates~$\IN(H)$. By the
  remark above and
  Lemma~\ref{lem:dominates}, we know that each of~$\IN(H_1)$ and~$\IN(H_2)$ is
  dominated by either~$a$ or~$b$.  Consequently, we may assume that
  $a$ dominates~$\IN(H_1)$ but not~$\IN(H_2)$ and $b$ dominates~$\IN(H_2)$ but not~$\IN(H_1)$.
  Let~$u_1 \in \IN(H_1)$ and~$u_2 \in \IN(H_2)$ be non-neighbors of~$b$ and~$a$, respectively.
  We distinguish two cases depending on the value of~$\mu$.

  \textbf{First case: $\mu \leq 2$.}
  Let~$F$ be the graph~$G-\IN(H)$ to which we add a crystal~$C(\mu)$ with
  poles~$a$ and~$b$. Let~$v_1,\dotsc,v_{\mu}$ be the inner vertices of this new
  crystal. Note that $\abs{V(G)} - \abs{V(F)} = k$. Since $d_H(a) \geq \wi{(H_1)} \geq
  \mu+2$ and similarly $d_H(b) \geq \mu +2$, the graph~$F$ has maximum degree at
  most~$\Delta$.
 
  By the minimality of~$G$ there is an equitable $k$-coloring~$\alpha$ of~$F$.
  Note that the restriction of~$\alpha$ to~$V(G) \setminus \IN(H)$ is also
  a proper partial coloring of~$G$. To equitably color~$G$, it suffices to extend
  this partial coloring to a proper coloring~$\beta$ of~$G$ such that the
  multiset~$\sst{\beta(v)}{v \in \IN(H)}$ equals the
  multiset~$C\coloneqq\{1,\dotsc,k\} \cup \sst{\alpha(v_i)}{1 \leq i \leq \mu}$.

  The colors~$\alpha(a)$ and~$\alpha(b)$ both have multiplicity exactly~$1$ in~$C$
  and the maximal multiplicity in~$C$ is at most~$\mu+1$.  
  We set~$\beta(u_1)\coloneqq\alpha(b)$ and~$\beta(u_2)\coloneqq\alpha(a)$.

  If the maximal multiplicity in~$C$ is~$2$, then
  Observation~\ref{obs:2parts} ensures that we can properly assign the~$k-2$
  remaining colors since each of~$H_1$ and~$H_2$ has at most~$k-3$ non-colored
  vertices. This yields an equitable $k$-coloring of~$G$, which is
  a contradiction.
  
  If the maximal multiplicity in~$C$ is~$3$, then $\mu = 2$, so
  $\wi(H_1) \geq k+2 \geq 4$.
  It follows then that $\IN(H_1)\setminus\{u_1\}$
  contains an independent set~$\{v_1,v_2\}$ of size~$2$. Indeed, otherwise
  $\IN(H_i)\setminus\{u_i\}$ would be a clique of size at least~$3$, which with
  $a$ or~$b$ would induce a copy of~$K_4$ in~$G$.
  Let~$v_3$ be a vertex in~$\IN(H_2)\setminus\{u_2\}$.  We color~$v_1$, $v_2$
  and~$v_3$ with the (unique) color of multiplicity~$3$ in~$C$.  Again,
  observation~\ref{obs:2parts} ensures that we can properly assign the~$k-3$
  remaining colors since each of~$H_1$ and~$H_2$ has at most~$k-4$ non-colored
  vertices.
    
  \textbf{Second case: $\mu \geq 3$.}
  Let~$F$ be the graph~$G -\IN(H)$ to which we add the edge~$ab$ if it is not
  already present. By the minimality of~$G$ there is an equitable
  $k$-coloring~$\alpha$ of~$F$. To equitably color~$G$, it suffices to
  extend~$\alpha$ to a proper coloring of~$G$ such that the
  multiset~$\sst{\alpha(v)}{v \in \IN(H)}$
  equals the multiset~$C\coloneqq\{1,\dotsc,k,1,\dotsc,\mu\}$.

  As $\mu\ge3$, every component of~$H-\{a,b\}$ has at least~$\mu+2\geq 5$
  vertices. Consequently, $a$ has two non-adjacent non-neighbors~$w_2$
  and~$w_2'$ in~$\IN(H_2)$. To see this, consider a component~$U$ of~$\IN(H_2)$. By
  Lemma~\ref{lem:independent}, the set~$U$ contains only one neighbor of~$a$. It follows
  that $\abs{U\setminus N(a)}\geq 3$, which gives the announced property since by
  Corollary~\ref{cor:treetree} the set~$U$ induces a tree in~$G$. One proves similarly that~$b$
  has two non-adjacent non-neighbors~$w_1$ and~$w_1'$ in~$\IN(H_1)$.  We
  set~$\alpha(w_2)\coloneqq\alpha(a)$,  $\alpha(w_1)\coloneqq\alpha(b)$ and if
  necessary $\alpha(w_2')\coloneqq\alpha(a)$
  and/or~$\alpha(w_1')\coloneqq\alpha(b)$.  After this, each of~$H_1$ and~$H_2$
  has at most~$k-3$ non-colored vertices. By Observation~\ref{obs:2parts}, we can
  extend this coloring using the $k-2$ remaining colors in~$C$.

  From now on, we assume that $a$ dominates~$\IN(H)$.
  Set~$F\coloneqq G-(\IN(H)\cup\{a\})$.
  By the minimality of~$G$ there is an equitable
  $k$-coloring~$\alpha$ of~$F$. To equitably color~$G$, it suffices to
  extend~$\alpha$ to a proper coloring of~$G$ such that the
  multiset~$\sst{\alpha(v)}{v \in \IN(H)\cup\{a\}}$
  equals the multiset~$C\coloneqq\{1,\dots,k+\mu+1\}$,
  where integers are reduced modulo~$k$.
  Note that $k+\mu+1\leq\wi(H_1)+\wi(H_2)+1<2k$ so every color has multiplicity
  either~$1$ or~$2$ in~$C$.
  
  The vertex~$a$ has at
  most~$k-3-\mu$ colored neighbors. There are $k-1-\mu$ colors with multiplicity
  one in~$C$. Consequently, it is possible to color~$a$ with a color of
  multiplicity one that is different from~$\alpha(b)$.

  We now place the color~$\alpha(b)$. We know that~$\wi(H)\ge k+\mu\ge 6$.
  By Lemma~\ref{lem:independent}, and since each component
  of~$H\setminus\{a,b\}$ has size at least $\mu+2\ge 3$, the vertex~$b$ has at
  least one non-neighbor in each of~$H_1$ and~$H_2$. We color a number of these
  non-neighbors equal to the multiplicity of~$\alpha(b)$ in~$C$ (which is either~$1$
  or~$2$) using the color~$\alpha(b)$.
  Observation~\ref{obs:2parts} then ensures
  that we can obtain an equitable coloring with the $k-2$ remaining colors.
\end{proof}

\begin{lemma}
  Let~$H\coloneqq S(H_1,H_2)$ be a graph represented by a construction subtree of~$T$. Assume that $\wi{(H_i)}\leq k-2$ for~$i \in \{1,2\}$.
  Then $\wi{(H)} \leq k-2$.
  \label{lem:series}
\end{lemma}
\begin{proof}
  Suppose, on the contrary, that $H$ contradicts the statement. Subject to this,
  we choose~$H$ to have as few vertices as possible.  We may assume that
  $\wi{(H)} \geq k+1$ by Lemmas~\ref{lem:tree:k-1} and~\ref{lem:tree:k}.
  Let~$b$ be the common pole of~$H_1$ and~$H_2$ and let~$a$ and~$c$ be the other
  poles of~$H_1$ and~$H_2$, respectively.

  \textbf{Case~1: For each~$i\in\{1,2\}$, the subgraph of~$G$ induced
  by~$\IN(H_i)$ contains an independent set~$\{u_i^1,u_i^2\}$ of size~$2$.}
  Let~$F$ be the graph~$G -\IN(H)$ to which we add the edge~$ac$ if it is not already
  present. By the minimality of~$G$ there is an equitable $k$-coloring~$\alpha$
  of~$F$, which we aim to extend to~$G$ such that the multiset~$\sst{\alpha(v)}{v
  \in \IN(H)}$ equals the multiset~$C\coloneqq\{1,\dotsc,\wi{(H)}\}$, where
  each integer is reduced modulo~$k$.

  We know that $\wi{(H)}\leq\wi{(H_1)}+\wi{(H_2)}+1\leq 2k-3$.  It
  follows that there is
  a color~$\gamma\in\{1,\dotsc,k\}\setminus\{\alpha(a),\alpha(c)\}$ of
  multiplicity one in~$C$.  We set~$\alpha(b)\coloneqq\gamma$,
  $\alpha(u_1^1)\coloneqq\alpha(b)$, $\alpha(u_1^2)\coloneqq\alpha(a)$ and if
  necessary~$\alpha(u_2^1)\coloneqq\alpha(b)$
  and/or~$\alpha(u_2^2)\coloneqq\alpha(a)$. For each~$i\in\{1,2\}$, the
  subgraph~$H_i$ has at most~$k-3$ non-colored vertices left, so by
  Observation~\ref{obs:2parts} it is possible to extend the coloring using
  the~$k-3$ remaining colors with the corresponding multiplicities.

  \textbf{Case~2: $\IN(H_1)$ induces a clique.}
  We know that
\[
  \wi(H_1)\geq\wi(H)-\wi(H_2)-1\geq k+1-(k-2)-1\geq2.
\]
  By Corollary~\ref{cor:treetree}, $\IN(H_1)$ is a forest, so $\wi(H_1)=2$.
  It forces moreover $\wi(H_2)$ to be~$k-2$.  This in particular implies that
  $k\ge4$.  Observe that the minimality of~$H$ ensures that each of the poles~$a$
  and~$c$ has at least two neighbors in~$H$.
    
  Let~$d$ and~$e$ be the inner vertices of~$H_1$. We define~$F$ to be the
  graph~$G -(\IN(H_1)\cup\IN(H_2))$ to which we add the edges~$ab$, $bc$
  and~$ac$ if not already present. Note that the graph thus obtained still has
  maximum degree at most~$\Delta$. By the minimality of~$G$ there is an equitable
  $k$-coloring~$\alpha$ of~$F$.

  It remains to deduce an equitable $k$-coloring of~$G$. To do so, we recolor~$b$ with
  a color~$\gamma$ different from~$\alpha(a)$, from~$\alpha(b)$ and
  from~$\alpha(c)$, which is possible as $k\ge4$. Next we color~$d$
  with~$\alpha(b)$ and $e$ with~$\alpha(c)$.  It now suffices to distribute
  arbitrarily the colors in~$\{1,\dotsc,k\}\setminus\{\gamma,\alpha(c)\}$ to the
  vertices in~$\IN(H_2)$.
\end{proof}

\noindent
We are now ready to conclude.
\begin{proof}[Proof of Theorem~\ref{thm:main}.]
  A direct induction on the tree~$T$ using Lemmas~\ref{lem:parallel}
  and~\ref{lem:series} shows that $G$ has at most~$k-2$ inner vertices.  This
  contradicts our assumption that $\abs{V(G)} > k$, thereby finishing the proof of
  Theorem~\ref{thm:main}.
\end{proof}



\end{document}